\documentclass{amsart}

\usepackage{amscd}
\usepackage{amsmath}
\usepackage{amssymb}
\usepackage{amsthm}
\usepackage{epsf}
\usepackage{latexsym}
\usepackage{verbatim}
\usepackage[all, cmtip]{xy}
\usepackage{tikz}
\usetikzlibrary{positioning}
\usetikzlibrary{matrix}

\tikzstyle{bsq}=[rectangle, draw, thick, minimum width=1cm, minimum height=1cm]
\tikzstyle{bver}=[rectangle, draw, thick, minimum width=1cm, minimum height=2cm]
\tikzstyle{bhor}=[rectangle, draw, thick, minimum width=2cm, minimum height=1cm]

\newtheorem{theorem}{Theorem}[section]
\newtheorem{lemma}[theorem]{Lemma}

\newtheorem{corollary}[theorem]{Corollary}
\newtheorem{proposition}[theorem]{Proposition}

\newtheorem{varexample}[theorem]{Example}
\newtheorem*{TropicalMRC}{Tropical Maximal Rank Conjecture}

\newtheorem*{MRC}{Maximal Rank Conjecture}

\theoremstyle{definition}
\newtheorem{remark}[theorem]{Remark}

\newtheorem{definition}[theorem]{Definition}

\newcommand{\PP}{\mathbb{P}}

\newcommand{\ZZ}{\mathbb{Z}}

\newcommand{\cA}{\mathcal{A}}

\newcommand{\ddiv}{\operatorname{div}}

\newcommand{\PL}{\operatorname{PL}}

\newcommand{\Sym}{\operatorname{Sym}}

\newenvironment{example}{\begin{varexample}
\begin{normalfont}}{\end{normalfont}
\end{varexample}}
\begin{document}
\title[Combinatorial and inductive methods]{Combinatorial and inductive methods for the tropical maximal rank conjecture}
\author{David Jensen}
\author{Sam Payne}
\date{}
\bibliographystyle{alpha}

\begin{abstract}
We produce new combinatorial methods for approaching the tropical maximal rank conjecture, including inductive procedures for deducing new cases of the conjecture on graphs of increasing genus from any given case.  Using explicit calculations in a range of base cases, we prove this conjecture for the canonical divisor, and in a wide range of cases for $m = 3$, extending previous results for $m = 2$.
\end{abstract}

\maketitle

\section{Introduction}

The classical maximal rank conjecture in algebraic geometry predicts the Hilbert function in each degree $m$ for the general embedding of a general algebraic curve of fixed genus $g$ and degree $d$ in projective space $\PP^r$, by specifying that certain linear multiplication maps on global sections should be of maximal rank.

\begin{MRC}
Suppose $g$, $r$, $d$, and $m$ are positive integers, with $r \geq 3$, such that $g > (r+1)(g-d+r)$, and let $X \subset \PP^r$ be a general curve of genus $g$ and degree $d$.  Then the multiplication map
\[
\mu_m: \Sym^m H^0(X, \mathcal{O}_X(1)) \rightarrow H^0(X, \mathcal{O}_X(m))
\]
is either injective or surjective.
\end{MRC}

In previous work, we introduced the notion of tropical independence to study ranks of such multiplication maps combinatorially, using minima of piecewise-linear functions on graphs arising via tropicalization.  As first applications, we gave a new proof of the Gieseker-Petri theorem \cite{tropicalGP}, and formulated a purely combinatorial analogue on tropical curves of the maximal rank conjecture for algebraic curves, which we proved for $m=2$ \cite{MRC}.

\begin{TropicalMRC}
Suppose $g$, $r$, $d$, and $m$ are positive integers, with $r \geq 3$, such that $g \geq (r+1)(g-d+r)$ and $d < g + r$.  Then there is a divisor $D$ of rank $r$ and degree $d$ whose class is vertex avoiding on a chain of $g$ loops with generic edge lengths, and a tropically independent subset $\cA \subset \{ \psi_I \ | \ \vert I \vert = m \}$ of size
\[
\vert \cA \vert = \min \left \{ {r + m \choose m}, \ md - g + 1 \right \}.
\]
\end{TropicalMRC}

\noindent Note that each case of the tropical maximal rank conjecture implies the classical maximal rank conjecture for the same parameters $g$, $r$, $d$, and $m$, through well-known tropical lifting and specialization arguments \cite[Proposition~4.7]{MRC}.

As the links to algebraic geometry are already established, the main purpose of this paper is to introduce new combinatorial methods for approaching the tropical maximal rank conjecture, and to use these methods to prove the conjecture for canonical divisors (i.e. the case where $r = g-1$ and $d = 2g-2$, for all $m$), and for a wide range of cases with $m = 3$.  Our results include inductive statements through which new cases of the tropical maximal rank conjecture can be deduced from other cases with smaller parameters (Theorems~\ref{thm:injectiveinduction} and \ref{thm:surjectiveinduction}), along with explicit combinatorial calculations in increasingly intricate examples (see Sections~\ref{Sec:Canonical}-\ref{Sec:SmallRank}), providing base cases for applying such inductions.

To state our results as cleanly as possible, we find it helpful to divide the space of parameters $(g,r,d,m)$ into the \emph{injective range}, where ${r + m \choose m} \leq md - g + 1$, and the \emph{surjective range}, where ${r + m \choose m} \geq md - g + 1$.
Under the hypotheses of the classical maximal rank conjecture, when $m>1$ the vector spaces $\Sym^m H^0(X, \mathcal{O}_X(1))$ and $H^0(X, \mathcal{O}_X(m))$ have dimension  ${r+m \choose m}$ and $md - g + 1$, respectively, so the injective range (resp. surjective range) is exactly the set of parameters for which the classical maximal rank conjecture predicts $\mu_m$ to be injective (resp. surjective).  In the setting of the tropical maximal rank conjecture, the set $\{ \psi_I \ | \ \# I = m \}$ has size ${r+m \choose m}$, so $(g,r,d,m)$ is in the injective or surjective range, respectively, according to whether the tropically independent set $\cA$ is supposed to consist of all possible $\psi_I$, or a subset of size $md - g + 1$.

We also find it convenient to change coordinates on the space of parameters in the conjectures, setting $s = g - d + r$ and $\rho  = g - (r+1)(g-d+r)$.  These new parameters are natural from the point of view of algebraic geometry; in the context of the classical maximal rank conjecture, $s = h^1(X, \mathcal{O}_X(1))$, and $\rho$ is the Brill-Noether number, which gives the dimension of the space of linear series of degree $d$ and rank $r$ on a general curve of genus $g$.  Note that $(r,s,\rho,m)$ uniquely determines $(g,r,d,m)$, by the formulas
\[
g = rs + \rho \mbox{ \ \ and \ \ } d = g + r - s.
\]
Also, when $d < g + r$ and $g \geq (r+1)(g-d+r)$, the parameters $s$ and $\rho$ satisfy $s \geq 1$ and $0 \leq \rho \leq g-1$.  We say that $(r,s,\rho,m)$ is in the injective (resp. surjective) range if the corresponding $(g,r,d,m)$ is in the injective (resp. surjective) range.

\begin{theorem} \label{thm:injectiveinduction}
Suppose $(r,s,\rho,m)$ is in the injective range.  Then the tropical maximal rank conjecture for $(r,s,\rho,m)$ implies the tropical maximal rank conjecture for $(r,s,\rho + 1, m)$ and $(r,s+1, \rho, m)$.
\end{theorem}

\begin{theorem} \label{thm:surjectiveinduction}
Suppose $r \geq s$ and $(r,s,\rho,m)$ is in the surjective range.  Then the tropical maximal rank conjecture for $(r,s,\rho,m)$ implies the tropical maximal rank conjecture for $(r+1, s, \rho, m)$.
\end{theorem}

Note that both of these inductive procedures increase the genus and keep $m$ fixed.  Also, if $(r,s,\rho,m)$ is in the injective range then so are ($r,s, \rho+1, m)$ and $(r, s+1, \rho, m)$.  Similarly, if $(r,s,\rho,m)$ is in the surjective range, then so is $(r+1, s, \rho, m)$.  Therefore, proving any single case of the tropical maximal rank conjecture (e.g. by explicit computation) yields infinitely many other cases of increasing genus.

For $m =3$, we prove the cases where $\rho = 0$ and either $s \geq r^2 / 4$ or $r = s + 1$ by explicit computation, and deduce the following result.

\begin{theorem} \label{thm:mainthm}
The tropical maximal rank conjecture holds for $(r,s,\rho,3)$ when
\begin{enumerate}
\item $s \geq r^2/4$, or
\item $\rho = 0$ and $r \geq s + 1$.
\end{enumerate}
\end{theorem}

Figure \ref{Fig:Range} illustrates the cases covered by Theorem \ref{thm:mainthm} when $\rho=0$.  One computes that $(r, s, 0, 3)$ is in the injective range exactly when
\[
s \geq \frac{(r+3)(r+2)(r+1)-6(3r+1)}{6(2r-1)},
\]
and the dotted curve represents the boundary between the injective range and the surjective range.   The cases covered by (1) are in the injective range, and the cases covered by (2) are in the surjective range.

\begin{figure}[h!]
\begin{tikzpicture}
\draw[->] (0,0)--(0,7.25) node[above] {$s$};
\draw[->] (0,0)--(7.25,0) node[right] {$r$};
\fill[fill=gray] (0,0) -- plot [domain=3:9] ({(\x-3)},{(\x-1)-1}) -- (6,0);
\fill[fill=gray] (6,0) -- plot [domain=6:7] ({\x},{7}) -- (7,0);
\draw[domain=3:9] plot ({(\x-3)},{(\x-1)-1});
\fill[fill=gray] (0,5/4) -- plot [domain=3:5.656] ({(\x-3)},{(\x^2/4)-1}) -- (0,7);
\draw[domain=3:5.656,smooth,variable=\x] plot ({(\x-3)},{(\x^2/4)-1});
\draw[domain=3:7.27,dotted,variable=\x] plot ({(\x-3)},{(((\x+3)*(\x+2)*(\x+1)-18*\x-6)/(12*\x-6))-1});
\draw (0,-0.5) node {3};
\draw (7,-0.5) node {10};
\draw (-0.5,0) node {1};
\draw (-0.5,7) node {8};
\draw (6,7.5) node {$s=r-1$};
\draw (2.656,7.5) node {$s=\frac{r^2}{4}$};
\end{tikzpicture}
\caption{Cases of the maximal rank conjecture covered by Theorem \ref{thm:mainthm} for $\rho=0$ and $m=3$.}
\label{Fig:Range}
\end{figure}
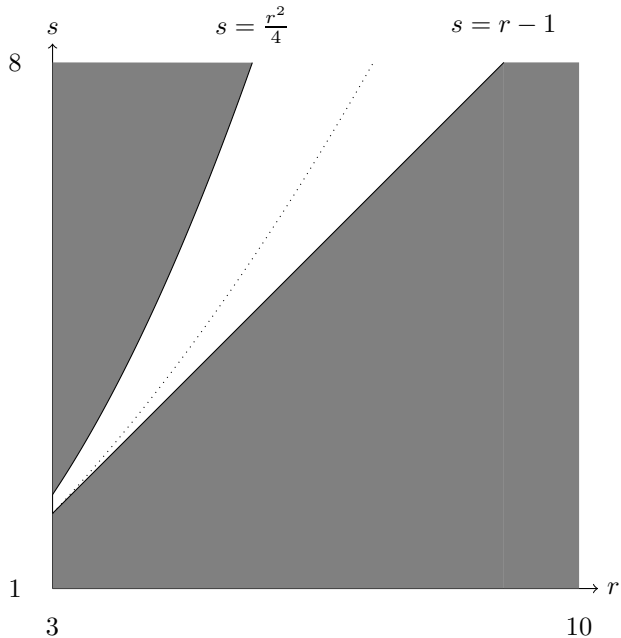

We can improve Theorem~\ref{thm:mainthm} by using explicit computations to prove additional cases in the unshaded region and then bootstrapping via our inductive methods.  Note in particular that, for fixed $r$ and $m$, Theorems~\ref{thm:injectiveinduction} and \ref{thm:surjectiveinduction} reduce the tropical maximal rank conjecture to finitely many cases.  For $m = 3$ and small $r$, we can carry out the necessary computations by hand.

\begin{theorem} \label{thm:smallrank}
The tropical maximal rank conjecture holds for $m = 3$ and $r \leq 4$.
\end{theorem}

As mentioned above, each case of the tropical maximal rank conjecture implies the corresponding case of the classical maximal rank conjecture.  The cases of the maximal rank conjecture in the injective range given by part (1) of Theorem~\ref{thm:mainthm} follow from Larson's maximal rank theorem for sections of curves \cite{Larson12}.  Indeed, Larson shows that, when $s \geq r^2/4$ and $X \subset \PP^r$ is general, any cubic hypersurface that contains a general hyperplane section of $X$ must contain the hyperplane.  It follows easily that $X$ is not contained in any cubic hypersurface, and hence $\mu_3$ is injective.

If $(g,r,d,m)$ is in the surjective range, then so is $(g,r,d,m+1)$, and the classical maximal rank conjecture for $(g,r,d,m)$ implies the maximal rank conjecture for $(g,r,d,m+1)$ (see, e.g. the proof of \cite[Theorem~1.2]{MRC}).  Unfortunately, we do not know how to prove the corresponding induction on $m$ for the tropical maximal rank conjecture in general.  Nevertheless, given that the maximal rank conjecture holds for $m =2$ and using the inductive statement for the classical maximal rank conjecture, it follows that the maximal rank conjecture holds for $(g,r,d,m)$, for all $m$, provided that it holds for $(g,r,d,3)$, and this is in the surjective range.

\begin{corollary}
The maximal rank conjecture holds for $(g,r,d,m)$, for all $m$, provided that either
\begin{enumerate}
\item $(g,r,d,3)$ is in the surjective range and $r \leq 4$, or
\item $\rho (g,r,d)= 0$ and $r \geq s+1$.
\end{enumerate}
\end{corollary}

\noindent These cases of the maximal rank conjecture appeared previously in \cite{Ballico09, Ballico12c}, over a field of characteristic zero; the present paper gives an independent and characteristic free proof.

The canonical divisor (i.e. the case where $r = g-1$ and $d = 2g-2$) is in the surjective range for all $m$, and in this special case we do manage to give an inductive argument starting from $m = 2$ to prove the tropical maximal rank conjecture for all $m$.

\begin{theorem} \label{thm:canonical}
Suppose $r = g-1$ and $d = 2g -2$.  Then the tropical maximal rank conjecture holds for $(g,r,d,m)$.
\end{theorem}

\noindent Equivalently, in terms of the parameters $(r,s,\rho,m)$, the tropical maximal rank conjecture holds for $s = 1$ and $\rho = 0$.

This paper, like many other recent tropical geometry papers such as \cite{BakerNorine07, BakerNorine09, HMY12, AMSW, ABKS, Len14}, is devoted to essentially combinatorial constructions on graphs, drawing inspiration and direction from analogous constructions on algebraic curves, with a view toward applications in algebraic geometry via lifting theorems and specialization.  The structure of the paper is as follows.  We briefly review the basic setup, with chains of loops, vertex avoiding divisors, and tropical independence of distinguished sections in Section \ref{Section:TheGraph}.  In Section \ref{Sec:Permit}, we present a mild generalization of the notion of permissible functions, which was one of the key combinatorial tools in \cite{MRC}.  In Section \ref{Sec:Induct}, we discuss inductive methods for tropical independence results, proving Theorems~\ref{thm:injectiveinduction} and \ref{thm:surjectiveinduction}.  We apply these methods first to the canonical divisor in Section~\ref{Sec:Canonical}, proving Theorem~\ref{thm:canonical}, and then to prove Theorem~\ref{thm:mainthm} in Section~\ref{Sec:Results}.  Finally, in Section \ref{Sec:SmallRank}, we prove enough additional cases for $m = 3$ and $r \leq 4$ by explicit computation to deduce Theorem~\ref{thm:smallrank}.

\medskip

\noindent \textbf{Acknowledgments.}  We have benefited from a number of helpful conversations with colleagues during the preparation of this work, and wish to thank, in particular, T.~Feng, C.~Fontanari, E.~Larson, and L.~Sauermann.  The work of DJ is partially supported by NSF DMS-1601896 and that of SP by NSF CAREER DMS-1149054.

\section{Combinatorics of the chain of loops}
\label{Section:TheGraph}

In this section, we recall the setup from \cite{MRC}, including definitions of all terms appearing in the statement of the tropical maximal rank conjecture.  The material of this section is developed in more detail in that paper and its precursors \cite{tropicalBN, tropicalGP}, to which we refer the reader for further details.

Let $\Gamma$ be a chain of loops with bridges, as pictured in Figure \ref{Fig:TheGraph}.  Note that $\Gamma$ has $2g+2$ vertices, one on the lefthand side of each bridge, which we label $w_0, \ldots , w_g$, and one on the righthand side of each bridge, which we label $v_1, \ldots, v_{g+1}$.  There are two edges connecting the vertices $v_k$ and $w_k$, the top and bottom edges of the $k$th loop, whose lengths are denoted $\ell_k$ and $m_k$, respectively, as shown schematically in Figure~\ref{Fig:TheGraph}.
\begin{figure}[h!]
\begin{tikzpicture}

\draw (-2.7,-.5) node {\footnotesize $w_0$};
\draw [ball color=black] (-2.5,-0.3) circle (0.55mm);
\draw (-2.5,-.3)--(-1.93,-0.25);
\draw [ball color=black] (-1.93,-0.25) circle (0.55mm);
\draw (-1.95,-0.5) node {\footnotesize $v_1$};
\draw (-1.5,0) circle (0.5);
\draw (-1,0)--(0,0.5);
\draw [ball color=black] (-1,0) circle (0.55mm);
\draw (-0.85,0.3) node {\footnotesize $w_1$};
\draw (0.7,0.5) circle (0.7);
\draw (1.4,0.5)--(2,0.3);
\draw [ball color=black] (1.4,0.5) circle (0.55mm);
\draw [ball color=black] (0,0.5) circle (0.55mm);
\draw (-0.2,0.75) node {\footnotesize $v_2$};
\draw (2.6,0.3) circle (0.6);
\draw (3.2,0.3)--(3.87,0.6);
\draw [ball color=black] (2,0.3) circle (0.55mm);
\draw [ball color=black] (3.2,0.3) circle (0.55mm);
\draw [ball color=black] (3.87,0.6) circle (0.55mm);
\draw (4.5,0.3) circle (0.7);
\draw (5.16,0.5)--(5.9,0);
\draw (6.4,0) circle (0.5);
\draw [ball color=black] (5.16,0.5) circle (0.55mm);
\draw (5.48,0.74) node {\footnotesize $w_{g-1}$};
\draw [ball color=black] (5.9,0) circle (0.55mm);
\draw [ball color=black] (6.9,0) circle (0.55mm);
\draw (5.7,-.2) node {\footnotesize $v_g$};
\draw (7.15,-.2) node {\footnotesize $w_g$};
\draw (6.9,0)--(7.5,0.2);
\draw [ball color=black] (7.5,0.2) circle (0.55mm);
\draw (7.94,0.1) node {\footnotesize $v_{g+1}$};

\draw [<->] (3.35,0.475)--(3.8,0.7);
\draw [<->] (3.3,0.4) arc[radius = 0.715, start angle=10, end angle=170];
\draw [<->] (3.3,0.2) arc[radius = 0.715, start angle=-9, end angle=-173];

\draw (3.5,0.725) node {\footnotesize$n_k$};
\draw (2.5,1.25) node {\footnotesize$\ell_k$};
\draw (2.75,-0.7) node {\footnotesize$m_k$};
\end{tikzpicture}
\caption{The graph $\Gamma$.}
\label{Fig:TheGraph}
\end{figure}
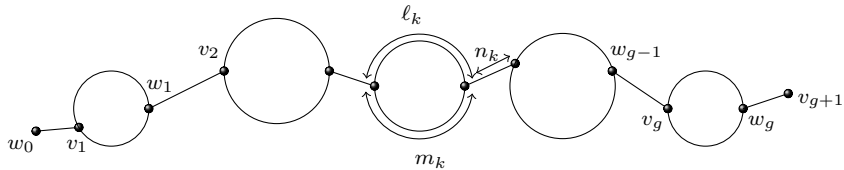
For $1 \leq k \leq g+1$ there is a bridge connecting $w_k$ and $v_{k+1}$, which we refer to as the $k$th bridge $\beta_k$, of length $n_k$.  Throughout, we assume that $\Gamma$ has admissible edge lengths in the following sense.

\begin{definition} \label{Def:Admissible}
The graph $\Gamma$ has \emph{admissible edge lengths} if
\[
4g m_k < \ell_k \ll \min \{ n_{k-1}, n_k\} \mbox{ for all $k$},
\]
there are no nontrivial linear relations $c_1 m_1 + \cdots + c_g m_g = 0$ with integer coefficients of absolute value at most $g + 1$, and
\begin{equation} \label{eq:blocks}
\sum_{i=\alpha s+1}^{(\alpha+1)s} \ell_i + \sum_{i=\alpha s+1}^{(\alpha+1)s-1} n_i \ll \min \{ n_{\alpha s}, n_{(\alpha+1)s} \},
\end{equation}
for every integer $\alpha \leq r$.
\end{definition}

\begin{remark}
The only difference between this notion of admissible edge lengths and \cite[Definition~4.1]{MRC} is the addition of the last condition (\ref{eq:blocks}), which can be thought of in the following way.  We divide the first $g-\rho$ loops of $\Gamma$ into $r+1$ \emph{blocks} consisting of $s$ loops each, such that the bridges separating these blocks are much longer than the blocks themselves.  This allows us to place additional restrictions on which functions can obtain the minimum at some point in a block, beyond those that come from each individual loop; see Section~\ref{Sec:Permit} and \cite[Section~6]{MRC}.  Figure \ref{Fig:Blocks} illustrates the decomposition of a chain of 12 loops into three blocks of four loops each.
\end{remark}

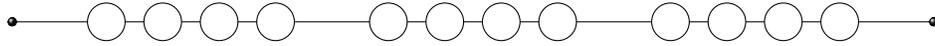
\begin{figure}[h!]
\begin{tikzpicture}
\matrix[column sep=0.5cm] {

\begin{scope}[grow=right, baseline]

\draw [ball color=black] (-1.25,-0) circle (0.55mm);
\draw (-1.25,0)--(-0.25,0);
\draw (0,0) circle (.25);
\draw (0.25,0)--(0.5,0);
\draw (0.75,0) circle (.25);
\draw (1,0)--(1.25,0);
\draw (1.5,0) circle (.25);
\draw (1.75,0)--(2,0);
\draw (2.25,0) circle (.25);

\draw (2.5,0)--(3.5,0);
\draw (3.75,0) circle (.25);
\draw (4,0)--(4.25,0);
\draw (4.5,0) circle (.25);
\draw (4.75,0)--(5,0);
\draw (5.25,0) circle (.25);
\draw (5.5,0)--(5.75,0);
\draw (6,0) circle (.25);

\draw (6.25,0)--(7.25,0);
\draw (7.5,0) circle (.25);
\draw (7.75,0)--(8,0);
\draw (8.25,0) circle (.25);
\draw (8.5,0)--(8.75,0);
\draw (9,0) circle (.25);
\draw (9.25,0)--(9.5,0);
\draw (9.75,0) circle (.25);
\draw (10,0)--(11,0);
\draw [ball color=black] (11,0) circle (0.55mm);

\end{scope}

\\};
\end{tikzpicture}
\caption{The graph $\Gamma$, with three blocks of four loops, when $g=12$, $r=2$, and $s=4$.}
\label{Fig:Blocks}
\end{figure}

Let $u_k$ be the midpoint of $\beta_k$, and decompose $\Gamma$ into locally closed subgraphs $\gamma_0, \ldots, \gamma_{g+1}$, as follows.  The subgraph $\gamma_0$ is the half-open interval $[w_0, u_0)$.  For $1 \leq i \leq g$, the subgraph $\gamma_i$, which includes the $i$th loop of $\Gamma$, is the union of the two half-open intervals $[u_{i-1}, u_i)$, which contain the top and bottom edges of the $i$th loop, respectively.  Finally, the subgraph $\gamma_{g+1}$ is the closed interval $[u_g, v_{g+1}]$.  We further write $\gamma_i^\circ$ for the $i$th embedded loop in $\Gamma$, which is a closed subset of $\gamma_i$, for $1 \leq i \leq g$.  The decomposition
\[
\Gamma = \gamma_0 \sqcup \cdots \sqcup \gamma_{g+1}
\]
is illustrated by Figure \ref{Fig:Decomposition}.  For $a \leq b$, we let $\Gamma_{[a,b]}$ be the locally closed, connected subgraph
\[
\Gamma_{[a,b]} = \gamma_a \sqcup \cdots \sqcup \gamma_b.
\]

\begin{figure}[h!]
\begin{tikzpicture}
\matrix[column sep=0.5cm] {

\begin{scope}[grow=right, baseline]

\draw (-6.75, -.3) node {$w_0$};
\draw [ball color=black] (-6.5,0) circle (0.55mm);
\draw (-6.5,0)--(-5.75,0);
\draw (-6,1.25) node {$\gamma_0$};
\draw [ball color=white] (-5.75,0) circle (0.55mm);

\draw (-5.3, -.3) node {$u_0$};
\draw [ball color=black] (-5.25,0) circle (0.55mm);
\draw (-5.25,0)--(-4.75,0);
\draw (-4,0) circle (.75);
\draw (-4,1.25) node {$\gamma_1$};
\draw (-3.25,0)--(-2.75,0);
\draw [ball color=white] (-2.75,0) circle (0.55mm);

\draw [ball color=black] (-2.25,0) circle (0.55mm);
\draw (-2.3, -.3) node {$u_1$};
\draw (-2.25,0)--(-1.75,0);
\draw (-1,0) circle (.75);
\draw (-1,1.25) node {$\gamma_2$};
\draw (-.25,0)--(.25,0);
\draw [ball color=white] (.25,0) circle (0.55mm);

\draw(.75, 0) node {$\cdots$};

\draw [ball color=black] (1.25,0) circle (0.55mm);
\draw (1.2, -.3) node {$u_{g-1}$};
\draw (1.25,0)--(1.75,0);
\draw (2.5,0) circle (.75);
\draw (2.5,1.25) node {$\gamma_g$};
\draw (3.25,0)--(3.75,0);
\draw [ball color=white] (3.75,0) circle (0.55mm);

\draw [ball color=black] (4.25,0) circle (0.55mm);
\draw (4.25, 0)--(5,0);
\draw (4.625,1.25) node {$\gamma_{g+1}$};
\draw [ball color=black] (5,0) circle (0.55mm);
\draw (5.25, -.3) node {$v_{g+1}$};

\end{scope}

\\};
\end{tikzpicture}
\caption{Decomposition of the graph $\Gamma$ into locally closed pieces $\{\gamma_k\}$.}
\label{Fig:Decomposition}
\end{figure}
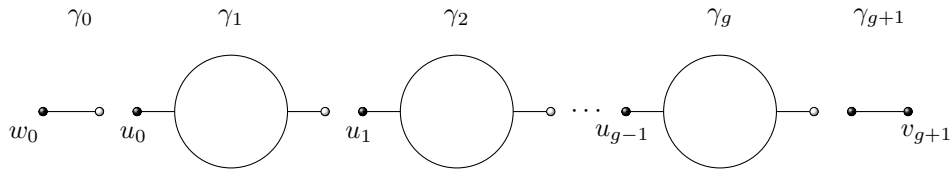

We write $\PL (\Gamma)$ for the set of continuous, piecewise linear functions on $\Gamma$ with integral slope.  For any divisor $D$ on $\Gamma$, we write
\[
R(D) : = \{ \psi \in \PL (\Gamma) \vert \ddiv \psi + D \geq 0 \}
\]
for the complete linear series of the divisor $D$.

The special divisor classes on $\Gamma$, i.e. the classes of degree $d$ and rank greater than $d-g$, are classified in \cite{tropicalBN}, where it is show that the Brill-Noether locus $W^r_d (\Gamma)$ parametrizing divisor classes of degree $d$ and rank $r$ is a union of $\rho$-dimensional tori.  These tori are in bijection with certain types of \emph{lingering lattice paths} in $\ZZ^r$.  These are sequences $p_0, \ldots, p_g$ starting and ending at
\[
p_0 = p_g = (r, r - 1, \ldots, 1) ,
\]
such that, for all $i$, $p_i - p_{i-1}$ is equal to 0, a standard basis vector $e_j$, or the vector $(-1, \ldots , -1)$, and satisfying
\[
p_i(0) > \cdots > p_i(r-1) > 0
\]
for all $i$.

The lingering lattice paths described above are in bijection with rectangular tableaux of size $(r+1) \times (g-d+r)$ with alphabet $1, \ldots , g$.  This bijection is given by placing $i$ in the $j$th column when $p_i - p_{i-1} = e_j$, and placing $i$ in the last column when $p_i - p_{i-1} = (-1,\ldots,-1)$.  When $p_i - p_{i-1} = 0$, the number $i$ is omitted from the tableau.

An open dense subset of the special divisor classes of degree $d$ and rank $r$ on $\Gamma$ consists of \emph{vertex avoiding} divisors.  We refer the reader to \cite[Definition~2.3]{LiftingDivisors} for a definition.  If $D$ is a divisor of rank $r$ on $\Gamma$ whose class is vertex avoiding, then there is a unique effective divisor $D_i \sim D$ such that $\deg_{w_0}(D_i) = i$ and $\deg_{v_{g+1}}(D_i) = r-i$.  Throughout, we will write $D$ for a $w_0$-reduced divisor on $\Gamma$ of degree $d$ and rank $r$ whose class is vertex avoiding, and $\psi_i$ for a piecewise linear function on $\Gamma$ such that $D + \ddiv (\psi_i) = D_i$.  Note that $\psi_i$ is uniquely determined up to an additive constant, and for $i<r$ the slope of $\psi_i$ along the bridge $\beta_j$ is $p_j(i)$.   The function $\psi_r$ is constant, so we set $p_j(r) = 0$ for all $j$.

For a multiset $I \subset \{ 0, \ldots, r \}$ of size $m$, let $D_I = \sum_{i \in I} D_i$ and let $\psi_I$ be a piecewise linear function such that $mD + \ddiv \psi_I = D_I$.  By construction, since $\ddiv \psi_I + mD = D_I$ is effective, the function $\psi_I$ is in $R(mD)$ and agrees with $\sum_{i \in I} \psi_i$ up to an additive constant.

Before stating our combinatorial conjecture, we recall the definition of tropical independence from \cite{tropicalGP}.

\begin{definition}
A set of piecewise linear functions $\{ \psi_0, \ldots, \psi_r \}$ on a metric graph $\Gamma$ is \emph{tropically dependent} if there are real numbers $b_0, \ldots, b_r$ such that for every point $v$ in $\Gamma$ the minimum
\[
\min \{\psi_0(v) + b_0, \ldots, \psi_r(v) + b_r \}
\]
occurs at least twice.  If there are no such real numbers then $\{ \psi_0, \ldots, \psi_r \}$ is \emph{tropically independent}.
\end{definition}

\begin{remark}
Since the functions $\psi_I$ are determined only up to an additive constant, we often suppress the constants $b_I$ in the definition of tropical dependence and assume that the minimum of the set $\{ \psi_I (v) \}$ occurs at least twice at every point $v \in \Gamma$.
\end{remark}

\section{Permissible functions}
\label{Sec:Permit}

Our strategy for proving cases of the tropical maximal rank conjecture will proceed by contradiction.  We choose a set $\cA$ of size $\min\{ {r+m \choose m}, md-g+1\}$ and let $\theta$ be the piecewise linear function
\[
\theta = \min_{I \in \cA} \{ \psi_I \},
\]
which is in $R(mD)$, with $\Delta$ the corresponding effective divisor
\[
\Delta = mD + \ddiv \theta.
\]
We will assume that the minimum occurs everywhere at least twice and use this to deduce properties of the function $\theta$ and the corresponding divisor $\Delta$, and ultimately obtain a contradiction.

For any function $\psi \in \PL (\Gamma)$, we let $\sigma_k \psi$ denote the slope of $\psi$ at $u_k$ going to the right.  So, for example, we have
\[
\sigma_k \psi_I = \sum_{i\in I} p_k (i).
\]
For ease of notation, we write
\[
\sigma_i = \sigma_i \theta \mbox{ and } \delta_i = \deg(\Delta|_{\gamma_i}).
\]

The nonnegative integer vector $\delta = (\delta_0, \ldots, \delta_{g+1})$ restricts the functions $\psi_I$ that can obtain the minimum on a given loop of $\Gamma$, as observed in \cite[Section~6]{MRC}.  Here we restate this observation in terms of the vector of slopes $\sigma = (\sigma_0 , \ldots , \sigma_{g+1})$, which makes the following definition and its basic properties particularly transparent.  We also give additional restrictions on functions that can obtain the minimum on a given block of loops, using condition (\ref{eq:blocks}) in Definition~\ref{Def:Admissible}.

\begin{definition}
Let $I \subset \{0, \ldots, r \}$ be a multiset of size $m$.  We say that $\psi_I$ is $\sigma$\emph{-permissible} on $\gamma_k^\circ$ if
\[
\sigma_{k-1} \psi_I \leq \sigma_{k-1} \mbox{ and } \sigma_k \psi_I \geq \sigma_k .
\]
Similarly, we say that $\psi_I$ is $\sigma$\emph{-permissible} on a block $\Gamma_{[\alpha s+1,(\alpha+1)s]}$ if
\[
\sigma_{\alpha s} \psi_I \leq \sigma_{\alpha s} \mbox{ and } \sigma_{(\alpha +1)s} \psi_I \geq \sigma_{(\alpha +1)s} .
\]
\end{definition}

Note that if a function $\psi_I$ is $\sigma$-permissible on a block, then it must be $\sigma$-permissible on some loop in that block.  On the other hand, $\psi_I$ may be $\sigma$-permissible on a loop without being $\sigma$-permissible on the block containing that loop.  The following lemma shows that both conditions are necessary for a function to obtain the minimum at some point of a loop or block.

\begin{lemma}
\label{Lem:LongBridges}
If $\psi_I (v) = \theta (v)$ for some $v \in \gamma_k^\circ$ then $\psi_I$ is $\sigma$-permissible on $\gamma_k^\circ$.  Similarly, if $\psi_I (v) = \theta (v)$ for some $v \in \Gamma_{[\alpha s+1,(\alpha+1)s]}$ then $\psi_I$ is $\sigma$-permissible on $\Gamma_{[\alpha s+1,(\alpha+1)s]}$.
\end{lemma}

\begin{proof}
The first statement is \cite[Lemma 6.2]{MRC}.  The second statement follows by the same argument, using the fact that the bridges between blocks are much longer than the blocks themselves.
\end{proof}

A consequence of Lemma \ref{Lem:LongBridges} is the following proposition, which controls the degree distribution of the divisor $\Delta$.

\begin{proposition}
\label{Prop:TwoChipsPerLoop}
Suppose that the number $t$ appears in the $i$th column of the tableau, and let $\beta$ be the minimum multiplicity of $i$ among multisets $I$ such that $\psi_I$ obtains the minimum at some point of $\gamma_t^{\circ}$.  Then $\delta_t \geq m-\beta$.
\end{proposition}

\begin{proof}
Note that the the degree $\delta_t$ can be determined directly from the slopes $\sigma_{t-1}$ and $\sigma_t$ along the bridges to the left and right of $\gamma_t$.  More precisely, we have
\[
\delta_t = \sigma_{t-1} - \sigma_t + m\deg (D \vert_{\gamma_t}) .
\]
It therefore suffices to show that
\[
\sigma_t \leq \sigma_{t-1} + m\deg (D \vert_{\gamma_t}) + \beta -m .
\]
To see this, let $I$ be a multiset such that $i$ has multiplicity $\beta$ in $I$, and $\psi_I$ obtains the minimum at some point of $\gamma_t^{\circ}$.  By Lemma \ref{Lem:LongBridges}, $\psi_I$ is $\sigma$-permissible, so $\sigma_t  \leq \sigma_t \psi_I$ and $\sigma_{t-1}\psi_I \leq \sigma_{t-1}$.  Note that, if $D$ has a chip on $\gamma_t$, then by definition $\sigma_t \psi_j = \sigma_{t-1} \psi_j$ for all $j \neq i$, whereas $\sigma_t \psi_i = \sigma_{t-1} \psi_i + 1$.  It follows that the slope of $\psi_I$ increases by $\beta$ from $u_{t-1}$ to $u_t$.  Similarly, if $D$ has no chips on $\gamma_t$, then the slope of $\psi_I$ decrease by $m-\beta$.  In other words,
\[
\sigma_t \psi_I = \sigma_{t-1} \psi_I + m\deg (D \vert_{\gamma_t}) + \beta -m,
\]
and the proposition follows.  \end{proof}

Proposition \ref{Prop:TwoChipsPerLoop} can be seen as a generalization of \cite[Proposition 5.2]{MRC}, which we reprove here.

\begin{corollary}
\label{Cor:TwoChips}
Suppose the minimum of $\{ \psi_I (v) \}_I$ occurs at least twice at every point $v$ in $\Gamma$.  Then $\delta_t \geq 2$ for all $t$.
\end{corollary}

\begin{proof}
Suppose the minimum occurs at least twice at every point in $\Gamma$.  Then we can choose $\psi_I$ and $\psi_J$ such that both obtain the minimum at $u_{t-1}$, and $\sigma_{t-1} \psi_I = \sigma_{t-1} \psi_J$.

We now assume $\delta_t \leq 1$ and proceed to find a contradiction.  By Proposition \ref{Prop:TwoChipsPerLoop}, since $\delta_t \leq 1$, the multisets $I$ and $J$, which have size $m$, must contain the value $i$ with multiplicity at least $m-1$.  In other words, $I = \{ i^{(m-1)} \alpha \}$ and $J = \{ i^{(m-1)} \beta \}$ for some $\alpha \neq \beta$.  However, $\sigma_{t-1} \psi_{\alpha} \neq \sigma_{t-1} \psi_{\beta}$, and hence $\sigma_{t-1} \psi_I \neq \sigma_{t-1} \psi_J$, which contradicts the choice of $\psi_I$ and $\psi_J$.
\end{proof}

\section{Inductive Methods}
\label{Sec:Induct}

In this section, we show how to deduce new cases of the tropical maximal rank conjecture, of increasing genus, from any given case.  This allows us to induct on the parameters and thereby prove the conjecture in a wide range of cases.

We first prove Theorem~\ref{thm:injectiveinduction}, which says that, when the tropical maximal rank conjecture holds for parameters $(r, s, \rho, m)$ in the injective range, then it also holds for $(r, s, \rho + 1, m)$ and $(r, s+1, \rho, m)$.

\begin{proof}[Proof of Theorem~\ref{thm:injectiveinduction}]
By assumption, there exists a divisor $D$ on the chain of $g$ loops of rank $r$ and degree $d$, whose class is vertex avoiding, such that the set of all functions $\psi_I$ is tropically independent.  We first show that we can increase $\rho$ by 1.  It suffices to construct a divisor $D'$ on the chain of $g+1$ loops of rank $r$ and degree $d+1$, whose class is vertex avoiding, such that the set of all functions $\psi_I$ is tropically independent.

We construct $D'$ by specifying that $D'\vert_{\Gamma_{[0,g]}} = D$, and the last step of the corresponding lattice path is lingering, with the point of $D'$ on $\gamma_{g+1}^{\circ}$ in sufficiently general position so that the class of $D'$ is vertex avoiding.  Then, the restrictions of the functions $\psi_I$ to $\Gamma_{[0,g]}$ are tropically independent, so the functions themselves are tropically independent as well.

We now show that we can increase $s$ by 1.  It suffices to construct a divisor $D'$ on a chain of $g+r+1$ loops of rank $r$ and degree $d+r$, whose class is vertex avoiding, such that the set of all functions $\psi_I$ is tropically independent.  As before, we construct $D'$ such that $D'\vert_{\Gamma_{[0,g]}} = D$, and now the last $r+1$ steps of the lingering lattice path are, in order, in each of the coordinate directions.  (This is equivalent to appending an extra row containing the numbers $g+1, \ldots , g+r+1$ to the bottom of the tableau corresponding to $D$, as shown in Figure~\ref{Fig:InductInj}.)
\begin{figure}[!h]
\begin{tikzpicture}
\matrix[column sep=0.7cm, row sep = 0.7cm] {
\begin{scope}[node distance=0 cm,outer sep = 0pt]
	      \node[bsq] (11) at (2.5,1) {1};
	      \node[bsq] (21) [below = of 11] {2};
	      \node[bsq] (12) [right = of 11] {3};
	      \node[bsq] (22) [below = of 12] {4};
	      \node[bsq] (13) [right = of 12] {5};
	      \node[bsq] (23) [below = of 13] {6};
	      \node[bsq] (14) [right = of 13] {7};
	      \node[bsq] (24) [below = of 14] {8};

          \draw [->] (4,-1)--(4,-1.5);
          %\draw (3.5,-1.25) node {\ref{thm:injectiveinduction}};

	      \node[bsq] (11b) at (2.5,-2.5) {1};
	      \node[bsq] (21b) [below = of 11b] {2};
	      \node[bsq] (31b) [below = of 21b] {9};
	      \node[bsq] (12b) [right = of 11b] {3};
	      \node[bsq] (22b) [below = of 12b] {4};
	      \node[bsq] (32b) [below = of 22b] {10};
	      \node[bsq] (13b) [right = of 12b] {5};
	      \node[bsq] (23b) [below = of 13b] {6};
	      \node[bsq] (33b) [below = of 23b] {11};
	      \node[bsq] (14b) [right = of 13b] {7};
	      \node[bsq] (24b) [below = of 14b] {8};
	      \node[bsq] (34b) [below = of 24b] {12};

\end{scope}
\\};
\end{tikzpicture}
\caption{The change in tableau when inducting on $s$ in Theorem~\ref{thm:injectiveinduction}.}
\label{Fig:InductInj}
\end{figure}
Again, the restrictions of the functions $\psi_I$ to $\Gamma_{[0,g]}$ are tropically independent, so the functions themselves are tropically independent as well.
\end{proof}

We now prove Theorem~\ref{thm:surjectiveinduction}, which says that, when $r \geq s$ and the tropical maximal rank conjecture holds for parameters $(r,s,\rho,m)$ in the surjective range, then it also holds for $(r+1, s, \rho, m)$.

\begin{proof}[Proof of Theorem~\ref{thm:surjectiveinduction}]
Let $g' = g+s$ and $d' = g+r+1$.  By assumption, there exists a divisor $D$ on the chain of $g$ loops of rank $r$ and degree $d$ and a tropically independent set $\cA$ of functions $\psi_I$ of size
\[
\vert \cA \vert = md-g+1 = (md'-g'+1)-((m-1)(s+1)+1) .
\]
We construct a divisor $D'$ on the chain of $g'$ loops of rank $r+1$ and degree $d'$ such that $D'\vert_{\Gamma_{[s+1,g']}} = D$, and the first $s$ steps in the lingering lattice path are all in the first coordinate direction.  (This is equivalent to adding $s$ to every entry of the tableau corresponding to $D$, and then appending an extra column containing the numbers $1, \ldots , s$ to the left of this tableau, as shown in Figure~\ref{Fig:InductSurj}.)
\begin{figure}[!h]
\begin{tikzpicture}
\matrix[column sep=0.7cm, row sep = 0.7cm] {
\begin{scope}[node distance=0 cm,outer sep = 0pt]
	      \node[bsq] (11) at (2.5,1) {1};
	      \node[bsq] (21) [below = of 11] {3};
	      \node[bsq] (12) [right = of 11] {2};
	      \node[bsq] (22) [below = of 12] {4};
	      \node[bsq] (13) [right = of 12] {5};
	      \node[bsq] (23) [below = of 13] {6};
	      \node[bsq] (14) [right = of 13] {7};
	      \node[bsq] (24) [below = of 14] {8};

          \draw [->] (6.5,0.5)--(7,0.5);
          %\draw (6.75,0.75) node {\ref{thm:surjectiveinduction}};

	      \node[bsq] (11a) at (8,1) {1};
	      \node[bsq] (21a) [below = of 11a] {2};
	      \node[bsq] (12a) [right = of 11a] {3};
	      \node[bsq] (22a) [below = of 12a] {5};
	      \node[bsq] (13a) [right = of 12a] {4};
	      \node[bsq] (23a) [below = of 13a] {6};
	      \node[bsq] (14a) [right = of 13a] {7};
	      \node[bsq] (24a) [below = of 14a] {8};
	      \node[bsq] (15a) [right = of 14a] {9};
	      \node[bsq] (25a) [below = of 15a] {10};

\end{scope}
\\};
\end{tikzpicture}
\caption{The change in tableau when inducting on $r$ Theorem \ref{thm:surjectiveinduction}.}
\label{Fig:InductSurj}
\end{figure}

We construct a set $\cA'$ of $md'-g'+1$ functions on the chain of $g'$ loops as follows.  First, replace each function $\psi_I \in \cA$ with $\psi_{I+1}$, where $I+1 = \{ i+1 \vert i \in I \}$.  Note that this is well defined, as we have increased the rank by 1.  Now, append to the set $\cA$ the function $\psi_{0^{(m)}}$ and all functions of the form $\psi_{0^{(k)}1^{(m-1-k)}\alpha}$, where $1 \leq k \leq m-1$ and $1 \leq \alpha \leq s+1$.  Note that, since $r \geq s$, all of these functions exist.  Moreover, this is precisely $(m-1)(s+1)+1$ functions, so the set $\cA'$ obtained by adding these functions has cardinality $md'-g'+1$.

Now, suppose that the minimum of the functions $\psi_I \in \cA'$ occurs everywhere at least twice.  On the bridge $\beta_s$, all $(m-1)(s+1)+1$ of the added functions have distinct slopes, and all have slope larger than $\sigma_s \psi_{1^{(m)}} = mr$, which is the largest possible slope among all functions $\psi_{I+1}$ for $\psi_I \in \cA$.  It follows that the only functions that may obtain the minimum to the right of $\beta_s$ are contained in the set $\cA$.  By assumption, however, the functions in $\cA$ are tropically independent on this subgraph, and the result follows.
\end{proof}

\begin{example}
Figures \ref{Fig:InductInj} and \ref{Fig:InductSurj} illustrate the change in tableaux for the inductive steps in Theorems~\ref{thm:injectiveinduction} and \ref{thm:surjectiveinduction}, starting from the case $(r,s,\rho,m) = (3,2,0,3)$ to deduce the cases $(3,3,0,3)$ and $(4,2,0,3)$, respectively.  This is a rare case where ${r+m \choose m} = md-g+1$, so it is in both the injective and surjective ranges (i.e. the maximal rank conjecture predicts $\mu_m$ to be an isomorphism), and hence both theorems can be applied.  Note, however, that some cases of the tropical maximal rank conjecture, such as $(r,s,\rho,m) = (4,3,0,3)$, cannot be deduced from any cases of smaller genus using Theorems~\ref{thm:injectiveinduction} and \ref{thm:surjectiveinduction}, so additional arguments are required to handle these base cases.  The case $(r,s,\rho,m) = (4,3,0,3)$ is proved in Theorem~\ref{thm:mainthm}.
\end{example}

%In only a few special cases, we can give inductive arguments proving new cases of the tropical maximal rank conjecture by induction on $m$.  See the proof of Theorem~\ref{thm:canonical}.  There is, however, a simple geometric argument showing that certain cases of the maximal rank conjecture imply the conjecture for larger values of $m$.

%\begin{proposition}
%\label{prop:allm}
%Let $X$ be a general curve of genus $g$.  Suppose $r \geq 3$, $\rho \geq 0$, and $d < g + r$.  Let $D_X$ be a general divisor on $X$ of degree $d$ and rank $r$, and suppose for some $m \geq 3$ the multiplication map
%\[
%\mu_m: \Sym^m \cL(D_X) \rightarrow \cL(mD_X)
%\]
%is surjective.  Then $\mu_{m'}$ is surjective for all $m' \geq m$.
%\end{proposition}

%\begin{proof}
%A proof of this is contained in the proof of \cite[Theorem~1.2]{MRC}.
%\end{proof}

%\noindent  In particular, within the surjective range, proving the tropical maximal rank conjecture for $(r, s, \rho, m)$ implies the classical maximal rank conjecture for algebraic curves for $(r,s,\rho, m')$ for all $m' \geq m$.

\section{The Canonical Divisor}
\label{Sec:Canonical}

Max Noether's theorem states that, if $X$ is a nonhyperelliptic curve and $D=K_X$ is the canonical divisor, then the maps $\mu_m$ are surjective for all $m$.  This can be seen as a strong form of the maximal rank conjecture in the special case of the canonical divisor.  The results of \cite[\S3]{MRC} provide a new proof of this result for sufficiently general curves, by proving the tropical maximal rank conjecture for $m = 2$ and using an algebraic geometry argument to deduce the classical maximal rank conjecture for $m > 2$.  In this section, we give a purely combinatorial proof of Max Noether's theorem for a general curve, by showing that the tropical maximal rank conjecture holds for all $m$, in the case of the canonical divisor.

\begin{proof}[Proof of Theorem~\ref{thm:canonical}]
Figure \ref{Fig:Canonical} illustrates the tableau corresponding to the canonical divisor.  By Theorem~\ref{thm:surjectiveinduction}, if the tropical maximal rank conjecture holds for $r=g-1$, $s=1$, and $\rho=0$ then it also holds for $r'=g$, $s=1$, and $\rho=0$.  The smallest genus of a nonhyperelliptic curve is $3$, so to prove the result by induction on $g$, it suffices to prove the base case $g=3$.

\begin{figure}[!h]
\begin{tikzpicture}
\matrix[column sep=0.7cm, row sep = 0.7cm] {
\begin{scope}[node distance=0 cm,outer sep = 0pt]
	      \node[bsq] (1) at (2.5,1) {1};
	      \node[bsq] (2) [right = of 1] {2};
	      \node[bsq] (3) [right = of 2] {3};
          \draw (5.5,1) node {$\cdots$};
   	      \node[bsq] (g) at (6.5,1) {$g$};

\end{scope}
\\};
\end{tikzpicture}
\caption{The tableau corresponding to the canonical divisor.}
\label{Fig:Canonical}
\end{figure}

Suppose $g=3$.  We must construct a set $\cA_m$ of $4m-2$ functions that are tropically independent on the chain of 3 loops.  We let $\cA_2$ be the set of all functions $\psi_{ij}$, and define $\cA_m$ recursively as follows.  For each function $\psi_I \in \cA_{m-1}$, we let $I'$ be the multiset obtained by adding an additional 1 to the multiset $I$, and include the function $\psi_{I'}$ in $\cA_m$.  We then add to this set the 4 functions
\[
\psi_{0^{(m)}}, \psi_{0^{(m-1)}2}, \psi_{02^{(m-1)}}, \psi_{2^{(m)}} .
\]

We first show that the functions $\psi_I \in \cA_m$ are tropically independent.  Suppose that the minimum $\theta = \min_{\psi_I \in \cA_m} \{ \psi_I \}$ occurs everywhere at least twice.  Then at the point $u_1$, the minimum must be obtained by two functions with the same slope.  Note that
\[
\sigma_1 \psi_{0^{(a)}1^{(b)}2^{(m-a-b)}} = 3a+b .
\]
We show, by induction on $m$, that the largest slope $\sigma_1 \psi_I$ that is obtained twice among the functions $\psi_I \in \cA_m$ is $m$.  To see this, first note that there is no multiset $I$ of size $m-1$ with $\sigma_1 \psi_I = 3m-4$.  It follows that $\psi_{0^{(m-1)}2}$ is the only function in $\cA$ with $\sigma_1 \psi_I = 3m-3$.  From this we see that the largest slope $\sigma_1 \psi_I$ that is obtained twice among the functions $\psi_I \in \cA_m$ is either $\sigma_1 \psi_{02^{(m-1)}} = 3$ or is obtained by two functions of the form $\psi_{I'}$ for $\psi_I \in \cA_{m-1}$.  The claim then follows by induction on $m$.  Thus, $\sigma_1 \leq m$.

Similarly, we have
\[
\sigma_2 \psi_{0^{(a)}1^{(b)}2^{(m-a-b)}} = 3a+2b .
\]
It follows that there is no multiset $I$ of size $m-1$ with $\sigma_2 \psi_I = 1$.  Therefore, $\psi_{02^{(m-1)}}$ is the only function in $\cA$ with $\sigma_2 \psi_I = 3$.  From this we see that the smallest slope $\sigma_2 \psi_I$ that is obtained twice among functions $\psi_I \in \cA_m$ is either $\sigma_2 \psi_{0^{(m-1)}2} = 3(m-1)$ or is obtained by two functions of the form $\psi_{I'}$ for $\psi_I \in \cA_{m-1}$.  It follows by induction on $m$ that the smallest slope $\sigma_2 \psi_I$ that is obtained twice among functions $\psi_I \in \cA_m$ is $2m$.  Thus, $\sigma_2 \geq 2m$.

We therefore see that $\sigma_2 \geq \sigma_1 + m$, and hence $\delta_2 \leq 0$.  But this contradicts Corollary~\ref{Cor:TwoChips}, which says that $\delta_t \geq 2$ for all $t$.
\end{proof}

\section{Proof of Theorem \ref{thm:mainthm}}
\label{Sec:Results}

In this section, we prove Theorem \ref{thm:mainthm}.  Throughout, our tableau will be the one in which the numbers $1,2, \ldots , s$ appear in the first column, $s+1 , s+2, \ldots , 2s$ appear in the second column, and so on.  We let $D$ be the corresponding divisor on the generic chain of loops.  Our goal is to find a set $\cA$ of functions $\psi_{ijk}$ of size $\vert \cA \vert = \min \{ {{r+3}\choose{3}}, 3d - g + 1 \}$ that are tropically independent.

After we choose the set $\cA$, we will suppose that the minimum
\[
\min_{\psi_{ijk} \in \cA} \{ \psi_{ijk} \}
\]
occurs everywhere at least twice, and let $\theta$ denote this minimum.  We let $\Delta = \ddiv (\theta) + 3D$ be the divisor corresponding to the function $\theta$.  By Corollary \ref{Cor:TwoChips}, we see that
\[
\delta_t := \deg (\Delta \vert_{\gamma_t} ) \geq 2 \mbox{ } \forall t .
\]
Equivalently, we have
\[
\sigma_t \leq \sigma_{t-1} + 1 \mbox{ } \forall t \leq rs .
\]
Moreover, equality can hold only if, for all functions $\psi_{ijk}$ obtaining the minimum on $\gamma_t$, at least one of $i,j,k$ is equal to $t$.

To show that the functions $\psi_{ijk} \in \cA$ are tropically independent, we will proceed from left to right across the graph, bounding the slope of $\theta$ on the long bridges $\beta_{\alpha s}$ between the blocks.  Before choosing the set $\cA$, we treat the first two blocks and the last two blocks separately, in the following two lemmas.

\begin{lemma}
\label{Lemma:FirstColumn}
If $r \leq s+1$, we have $\sigma_s \leq 3r+s-4$.  Similarly, $\sigma_{rs} \geq 2s+4$.
\end{lemma}

\begin{proof}

Since $r \leq s+1$, every function of the form $\psi_{00\alpha}$ has larger slope on $\beta_s$ than any function not of this form.  It follows that the largest slope that is obtained by two or more of the functions $\psi_{ijk}$ on $\beta_s$ is
\[
3r+s-4 = (r+s) + (r-1) + (r-3) = (r+s) + (r-2) + (r-2) .
\]
It follows that $\sigma_s \leq 3r+s-4$.  A symmetric argument shows that $\sigma_{rs} \geq 2s+4$.
\end{proof}

\begin{lemma}
\label{Lemma:SecondColumn}
If $r \leq s+1$, we have $\sigma_{2s} \leq 3r+s-5$.  Similarly, $\sigma_{(r-1)s} \geq 2s+5$.
\end{lemma}

\begin{proof}

We first show that $\sigma_{2s} \leq 3r+s-4$.  We do this in two cases.  First, suppose that, for each function $\psi_{ijk}$ obtaining the minimum along $\beta_s$, at least one of $i,j,k$ is equal to 1.  Since the minimum occurs at least twice, there must be two such functions that have the same slope at $u_s$.  We see that the largest slope that can be obtained more than once by functions involving 1 is
\[
3r-5 = \sigma_s \psi_{113} = \sigma_s \psi_{122} .
\]
Since $\sigma_t \leq \sigma_{t-1} + 1$ for all $t$, it follows that $\sigma_{2s} \leq 3r+s-5$.

Now, suppose conversely that the minimum at some point of $\beta_s$ is obtained by a function $\psi_{ijk}$ where none of $i,j,k$ is equal to 1.  Then, since $\sigma_s \psi_{ijk} = \sigma_{2s} \psi_{ijk}$, by Lemma \ref{Lem:LongBridges}, we have $\sigma_{2s} \leq \sigma_s$.  By Lemma \ref{Lemma:FirstColumn}, it follows that $\sigma_{2s} \leq 3r+s-4$.

To see that this inequality is strict, note that there is only one function $\psi_{ijk}$ with $\sigma_{2s} \psi_{ijk} = 3r+s-4$, namely $\psi_{022}$.  It follows that $\sigma_{2s} < 3r+s-4$.  That $\sigma_{(r-1)s} \geq 2s+5$ follows by a symmetric argument.
\end{proof}

We now prove Theorem \ref{thm:mainthm} in the injective case.

\begin{proof}[Proof of part (1) of Theorem~\ref{thm:mainthm}]
Our goal is to show that, if $s \geq \frac{r^2}{4}$, then the full set of functions $\psi_{ijk}$ is tropically independent.  We show that $\sigma_{\alpha s} \leq 3r+s-2\alpha$ for all $\alpha \leq r$.  We prove this by induction on $\alpha$, the cases $\alpha = 1,2$ being Lemmas \ref{Lemma:FirstColumn} and \ref{Lemma:SecondColumn}.  Suppose that $\sigma_{\alpha s} \leq 3r+s-2\alpha$, and for contradiction assume that $\sigma_{(\alpha +1)s} > 3r+s-2\alpha-2$.  Then any function $\psi_{ijk}$ obtaining the minimum on $\Gamma_{[\alpha s+1,(\alpha+1)s]}$ must satisfy
\[
\sigma_{\alpha s} \psi_{ijk} \leq 3r+s-2\alpha \mbox{ and } \sigma_{(\alpha+1)s} \psi_{ijk} > 3r+s-2\alpha-2 .
\]
Note that, if none of $i,j,k$ are equal to $\alpha$, then either
\[
\sigma_{\alpha s} \psi_{ijk} \geq 2(r-\alpha+s+1) > 3r+s-2\alpha,
\]
(if at least two of $i,j,k$ are smaller than $\alpha$), or
\[
\sigma_{(\alpha+1)s} \psi_{ijk} \leq r+s+2(r-\alpha-1) = 3r+s-2\alpha-2
\]
(if at least two of $i,j,k$ are greater than $\alpha$).

It follows that at least one of $i,j,k$ must be equal to $\alpha$.  Moreover, since $3r+s-2\alpha < 3r+2s-3\alpha+2 = \sigma_{\alpha s} \psi_{(\alpha-1)^2\alpha}$, at most one of $i,j,k$ can be smaller than $\alpha$.  Similarly, since $3r+s-2\alpha-2 > 3r+s-3\alpha-2 = \sigma_{(\alpha+1)s} \psi_{\alpha(\alpha+1)^2}$, at most one of $i,j,k$ can be larger than $\alpha$.  In other words, for each function $\psi_{ijk}$ that obtains the minimum at some point of $\Gamma_{[\alpha s+1,(\alpha+1)s]}$, we may assume after reordering that $i \leq \alpha$, $j = \alpha$, and $k \geq \alpha$.  Note that the number of such triples is $(\alpha +1)(r+1-\alpha)$, which is maximized when $\alpha = \frac{r}{2}$, in which case it is equal to $\frac{(r+2)^2}{4}$.

We now show that the restrictions of these functions to $\Gamma_{[\alpha s+1,(\alpha+1)s]}$ are tropically independent.  Note that all of these functions contain $\psi_j = \psi_{\alpha}$ as a summand, so it suffices to show that the functions $\psi_{ik} = \psi_{ijk} - \psi_j$ are tropically independent.  We write
\[
\theta_{\alpha} = \min \{ \psi_{ik} \vert_{\Gamma_{[\alpha s+1,(\alpha+1)s]}} \} .
\]

For $\sigma_{\alpha s} \psi_{\alpha r} = r-\alpha+s \leq \sigma \leq (r+s)+(r-\alpha) = \sigma{(\alpha+1)s} \psi_{0\alpha}$, let $\Gamma_{\sigma}$ denote the union of the loops $\gamma_t \subset \Gamma_{[\alpha s+1,(\alpha+1)s]}$ for which $\sigma_{t-1} \theta_{\alpha} = \sigma$.  If no such $t$ exists, let $\Gamma_{\sigma}$ be a segment of the bridge between $\Gamma_{\sigma+1}$ and $\Gamma_{\sigma-1}$.  Let $\cA_{\sigma}$ be the set of functions $\psi_{ik}$ that are permissible on $\Gamma_{\sigma}$.  If $\Gamma_{\sigma}$ has positive genus $g(\Gamma_{\sigma})$, then by a minor variant of \cite[Proposition 7.6]{MRC}, we see that
\[
\vert \cA_{\sigma} \vert > g(\Gamma_{\sigma})+1.
\]
By applying Corollary \ref{Cor:TwoChips} in the case $m=2$, we see that the slopes $\sigma_t \theta_{\alpha}$ do not increase.  Since the slopes $\sigma_t \psi_{ik}$ do not decrease, we see that $\cA_{\sigma} \cap \cA_{\sigma'} = \emptyset$ for $\sigma \neq \sigma'$.  Moreover, by considering the functions $\psi_{ik}$ where neither $i$ nor $k$ is equal to $\alpha$, we see that $\cA_{\sigma} \neq \emptyset$ for all $\sigma$.  It follows that
\[
\vert \bigcup_{\sigma = r-\alpha+s}^{(r+s)+(r-\alpha)} \cA_{\sigma} \vert \geq \sum_{\sigma = r-\alpha+s}^{(r+s)+(r-\alpha)} \vert \cA_{\sigma} \vert > \sum_{\sigma = r-\alpha+s}^{(r+s)+(r-\alpha)} (g(\Gamma_{\sigma})+1) = s+r+1.
\]
Note that this inequality is strict.  By assumption, however, we have
\[
s+r+1 \geq \frac{r^2}{4} + (r+1) = \frac{(r+2)^2}{4},
\]
a contradiction.

We therefore see that $\sigma_{rs} \leq r+s$.  But by Proposition \ref{Prop:TwoChipsPerLoop} we have
\[
\sigma_t \leq \sigma_{t-1} - 2 \mbox{ for } rs < t \leq (r+1)s ,
\]
so $\sigma_{(r+1)s} \leq r-s < 0$, a contradiction.
\end{proof}

And now we prove Theorem \ref{thm:mainthm} in the surjective case.

\begin{proof}[Proof of part (2) of Theorem~\ref{thm:mainthm}]
We first consider the case $s=r-1$, $\rho=0$, $m=3$.  Note that
\[ 3d-g+1 = 2g+3r-3s+1
\]
\[
= 2(r^2 + 1) = {{r+3}\choose{3}} - {{r-1}\choose{3}} .
\]
We let $\cA$ be the set of functions $\psi_{ijk}$ such that at least one of $i,j,k$ is equal to 0, 1, $r-1$, or $r$.   The above computation shows that $\vert \cA \vert = 3d-g+1$, as desired.  It suffices to show that the functions $\psi_{ijk} \in \cA$ are tropically independent.

We will show that, for all $\alpha$ in the range $1 \leq \alpha < r-1$, we have $\sigma_{(\alpha+1)s} \leq 4r - 4 - 2\alpha$.  As a consequence, we see that $\sigma_{(r-1)s} \leq 2r$, contradicting the second part of Lemma \ref{Lemma:SecondColumn}.  We prove this by induction on $\alpha$, the case $\alpha=1$ being the first part of Lemma \ref{Lemma:SecondColumn}.

We first show that $\sigma_{(\alpha+1)s} \leq 4r - 2 - 2\alpha$.  By induction, we know that $\sigma_{\alpha s} \leq 4r-2-2\alpha$.  Recall that the slope of $\theta$ may increase by at most 1 from one bridge to the next.  If $\sigma_{(\alpha+1)s} > 4r-2-2\alpha$, then there must be a $t$ such that
\[
\sigma_{\alpha s+t} = \sigma_{\alpha s+t-1} + 1 = 4r-1-2\alpha .
\]
By Lemma \ref{Lem:LongBridges}, all of the functions $\psi_{ijk}$ obtaining the minimum on $\gamma_t^{\circ}$ must have at least one of $i,j,k$ equal to $\alpha$, and since no 2 functions have identical restrictions to $\gamma_t^{\circ}$, by \cite[Lemma~5.1]{MRC} there must be at least 3 such functions.  By construction of the set $\cA$, however, no such set of 3 functions exists.  It follows that $\sigma_{(\alpha+1)s} \leq 4r - 2 - 2\alpha$.

Now, we note that there is only one function $\psi_{ijk}$ with $\sigma_{(\alpha+1)s} \psi_{ijk} = 4r-3-2\alpha$, and only one with $\sigma_{(\alpha+1)s} \psi_{ijk} = 4r-2-2\alpha$.  To see this, note that if $j$ and $k$ are greater than $\alpha$, then
\[
\sigma_{(\alpha+1)s} \psi_{ijk} \leq (2r-1) + 2(r-\alpha-1) = 4r-\alpha-3 ,
\]
with equality if and only if $i=0$, $j=k=\alpha+1$.  Similarly, if $i$ and $j$ are less than or equal to $\alpha$, then
\[
\sigma_{(\alpha+1)s} \psi_{ijk} \geq 2(2r-\alpha-1) = 4r-2\alpha-2 ,
\]
with equality if and only if $i=j=\alpha$, $k=r$.  It follows that $\sigma_{(\alpha+1)s} \leq 4r - 4 - 2\alpha$.

Finally, by Theorem~\ref{thm:surjectiveinduction}, we see that the tropical maximal rank conjecture therefore holds for $r \geq s+1$, $\rho=0$, and $m=3$ by induction on $r$.
\end{proof}

%Using an algebraic geometry argument (Proposition~\ref{prop:allm}), we conclude that the classical maximal rank conjecture holds for all $m$, when $\rho =0$ and $d>g$.  Note that the case $m=2$ is known already by \cite{MRC}.  The following theorem establishes the conjecture in these cases for the remaining values of $m$.

%\begin{theorem}
%Let $X$ be a smooth projective curve of genus $g$ over a nonarchimedean field such that the minimal skeleton of the Berkovich analytic space $X^{\an}$ is isometric to $\Gamma$.  Suppose $r \geq 3$, $\rho = 0$, and $g < d < g + r$.  Then there is a very ample complete linear series $\cL(D_X)$ of degree $d$ and rank $r$ on $X$ such that the multiplication maps
%\[
%\mu_m: \Sym^m \cL(D_X) \rightarrow \cL(mD_X)
%\]
%are surjective for all $m \geq 3$.
%\end{theorem}

%\begin{proof}
%By Theorem \ref{thm:mainthm}, the map $\mu_3$ is surjective for a general divisor $D_X$ of degree $d$ and rank $r$ on $X$.  By Proposition \ref{prop:allm}, therefore, the map $\mu_m$ is surjective for all $m \geq 3$.
%\end{proof}

\section{Divisors of Small Rank}
\label{Sec:SmallRank}

A consequence of Theorem~\ref{thm:injectiveinduction} is that, for fixed $r$ and $m$, it suffices to prove the tropical maximal rank conjecture for finitely many $s$ and $\rho$.  In this section, we use this observation to prove the tropical maximal rank conjecture for $m = 3$ and $r \leq 4$.  We also prove the case $m=3$, $r=5$, $\rho=0$.  We hope that the examples in this section will illuminate some of the additional complexities that arise when we move beyond the cases explored in the earlier sections, while simultaneously suggesting that the tropical maximal rank conjecture should hold far more generally.

\subsection{Rank 3}  Fix $r=3$ and $m=3$.  By Theorem \ref{thm:mainthm}, the tropical maximal rank conjecture holds for $s=2$ and $\rho=0$.  Since
\[
{{3+3}\choose{3}} = 3 \cdot 9 - 8 + 1,
\]
we can use Theorem~\ref{thm:injectiveinduction} to conclude that the tropical maximal rank conjecture holds for all $s \geq 2$, $\rho \geq 0$.  It therefore suffices to consider the cases where $s=1$.  When $s=1$ and $\rho=0$, the tropical maximal rank conjecture holds by Theorem \ref{thm:canonical}.

Let us consider the case $s=1, \rho=1$ in detail.  In this case, consider the tableau pictured in Figure \ref{Fig:RhoOne}, and let $\cA$ be the set of all functions $\psi_{ijk}$ other than
\[
\psi_{003}, \psi_{023}, \psi_{033}.
\]
We show that the set $\cA$ is tropically independent.  To see this, suppose that the minimum $\min \{ \psi_{ijk} \}$ occurs everywhere at least twice, and let $\theta$ denote this minimum.  The largest slope $\sigma_1 \psi_{ijk}$ that is obtained at least twice among functions $\psi_{ijk} \in \cA$ is
\[
\sigma_1 \psi_{013} = \sigma_1 \psi_{022} = \sigma_1 \psi_{111} = 6.
\]
It follows that $\sigma_1 \leq 6$.  If $\sigma_2 \geq 6$, then the three functions listed above are precisely the $\sigma$-permissible functions on $\gamma_2^{\circ}$.  One can check, however, that the restrictions of these three functions to $\gamma_2^{\circ}$ are tropically independent, and thus $\sigma_2 < 6$.  We then see that, if $\psi_{ijk} \in \cA$ obtains the minimum at some point of $\Gamma_{[3,5]}$, then none of $i,j,$ or $k$ is equal to 0.  That the restrictions of these functions to $\Gamma_{[3,5]}$ are tropically independent follows from Theorem \ref{thm:canonical}.

\begin{figure}[h!]
\begin{tikzpicture}
\matrix[column sep=0.7cm, row sep = 0.7cm] {
\begin{scope}[node distance=0 cm,outer sep = 0pt]
	      \node[bsq] (11) at (2.5,1) {1};
	      \node[bsq] (12) [right = of 11] {3};
	      \node[bsq] (13) [right = of 12] {4};
	      \node[bsq] (14) [right = of 13] {5};

\end{scope}
\\};
\end{tikzpicture}
\caption{The case $r=3$, $s=1$, $\rho=1$.}
\label{Fig:RhoOne}
\end{figure}

The cases where $\rho=2$ or $\rho=3$ follow by a similar argument.  In the first case, we consider the tableau depicted in Figure \ref{Fig:RhoTwo}, and let $\cA$ be the set of all functions $\psi_{ijk}$ other than $\psi_{003}$.  In the second case, we consider the tableau depicted in Figure \ref{Fig:RhoThree} and let $\cA$ be the full set of functions $\psi_{ijk}$.  Finally, using Theorem~\ref{thm:injectiveinduction} to argue by induction from the base case $\rho=3$, we see that the tropical maximal rank conjecture holds for $s=1$ and all $\rho \geq 3$.  We therefore see that the maximal rank conjecture for cubics holds when $r=3$.

\begin{figure}[h!]
\begin{tikzpicture}
\matrix[column sep=0.7cm, row sep = 0.7cm] {
\begin{scope}[node distance=0 cm,outer sep = 0pt]
	      \node[bsq] (11) at (2.5,1) {1};
	      \node[bsq] (12) [right = of 11] {3};
	      \node[bsq] (13) [right = of 12] {5};
	      \node[bsq] (14) [right = of 13] {6};

\end{scope}
\\};
\end{tikzpicture}
\caption{The case $r=3$, $s=1$, $\rho=2$.}
\label{Fig:RhoTwo}
\end{figure}

\begin{figure}[h!]
\begin{tikzpicture}
\matrix[column sep=0.7cm, row sep = 0.7cm] {
\begin{scope}[node distance=0 cm,outer sep = 0pt]
	      \node[bsq] (11) at (2.5,1) {1};
	      \node[bsq] (12) [right = of 11] {4};
	      \node[bsq] (13) [right = of 12] {6};
	      \node[bsq] (14) [right = of 13] {7};

\end{scope}
\\};
\end{tikzpicture}
\caption{The case $r=3$, $s=1$, $\rho=3$.}
\label{Fig:RhoThree}
\end{figure}

\subsection{Rank 4}  Now fix $r=4$ and $m=3$.  By part (1) of Theorem \ref{thm:mainthm}, the tropical maximal rank conjecture holds for $s \geq 4$ and all $\rho \geq 0$.  Similarly, by part (2) of Theorem \ref{thm:mainthm}, the tropical maximal rank conjecture holds for $s \leq 3$ and $\rho=0$.  So it suffices to consider the cases where $s \leq 3$ and $\rho > 0$.  Moreover, if the tropical maximal rank conjecture holds for some pair $(s,\rho)$ satisfying $7s+2\rho \geq 22$, then by Theorem~\ref{thm:injectiveinduction}, it holds for all larger values of $s$ and $\rho$.  It therefore suffices to consider the following pairs:
\[
(1,1), (1,2), (1,3), (1,4), (1,5), (1,6), (1,7), (1,8), (2,1), (2,2), (2,3), (2,4), (3,1) .
\]
To examine each of these cases individually would be somewhat tedious, so we will focus on just one, the case where $s=3$ and $\rho=1$.  In this case, consider the tableau pictured in Figure \ref{Fig:Rank4}.  We adjust the edge lengths so that the middle block has genus 4.  That is, rather than setting the bridges $\beta_t$ to be longer when $t$ is a multiple of 3, we instead specify the bridges $\beta_3, \beta_6, \beta_{10},$ and $\beta_{13}$ to be longer than the others.  Suppose that the minimum $\min \{ \psi_{ijk} \}$ occurs everywhere at least twice, and let $\theta$ denote this minimum.
\begin{figure}[h!]
\begin{tikzpicture}
\matrix[column sep=0.7cm, row sep = 0.7cm] {
\begin{scope}[node distance=0 cm,outer sep = 0pt]
	      \node[bsq] (11b) at (2.5,0) {1};
	      \node[bsq] (21b) [below = of 11b] {2};
	      \node[bsq] (31b) [below = of 21b] {3};
	      \node[bsq] (12b) [right = of 11b] {4};
	      \node[bsq] (22b) [below = of 12b] {5};
	      \node[bsq] (32b) [below = of 22b] {6};
	      \node[bsq] (13b) [right = of 12b] {7};
	      \node[bsq] (23b) [below = of 13b] {8};
	      \node[bsq] (33b) [below = of 23b] {9};
	      \node[bsq] (14b) [right = of 13b] {11};
	      \node[bsq] (24b) [below = of 14b] {12};
	      \node[bsq] (34b) [below = of 24b] {13};
	      \node[bsq] (15b) [right = of 14b] {14};
	      \node[bsq] (25b) [below = of 15b] {15};
	      \node[bsq] (35b) [below = of 25b] {16};
\end{scope}
\\};
\end{tikzpicture}
\caption{The case $r=4$, $m=3$, $s=3$, $\rho=1$.}
\label{Fig:Rank4}
\end{figure}

Following the proof of part (1) of Theorem \ref{thm:mainthm}, we see that $\sigma_3 \leq 13$ and consequently $\sigma_6 \leq 11$.  As in the same proof, if $\sigma_{10} > 9$, then the only $\sigma$-permissible functions on the block $\Gamma_{[7,10]}$ are of the form $\psi_{ij2}$, where $i \leq 2$ and $j \geq 2$.  Note that there are 9 such functions.  This is one larger than the bound $s+r+1$ obtained in the proof of Theorem \ref{thm:mainthm}, but because we have also increased the genus of the block by one, the same argument shows that these 9 functions have tropically independent restrictions to the block.  It follows that $\sigma_{10} \leq 9$. Continuing to follow the proof of Theorem \ref{thm:mainthm}, we see that $\sigma_{13} \leq 7$, and thus $\sigma_{16} \leq 1$.  This is a contradiction, because there is only one function $\psi_{ijk}$ with $\sigma_{16} \psi_{ijk} =1$, and only one with $\sigma_{16} \psi_{ijk} = 0$.

In a similar way, we can prove the tropical maximal rank conjecture for $r=4$, $m=3$, and all $s$ and $\rho$.

\subsection{Rank 5}

Now fix $r=5$ and $m=3$.  This is the first case where not all of the $\rho=0$ cases are covered by Theorem~\ref{thm:mainthm}.  We will show that the tropical maximal rank conjecture for cubics holds when $r=5$ and $\rho =0$.   More specifically, part (2) of Theorem \ref{thm:mainthm} shows that the tropical maximal rank conjecture holds for $s \leq 4$ and $\rho = 0$, whereas part (1) shows that the tropical maximal rank conjecture for cubics holds for $s \geq 7$ and all $\rho \geq 0$.  We now consider the case where $r=s=5$ and $\rho=0$.  Note that, by Theorem \ref{thm:injectiveinduction}, this implies the cases where $r=5$, $s \geq 5$, and $\rho \geq 0$.  In particular, it implies the remaining $\rho=0$ case, that of $s=6$.

Again, we consider the rectangular tableau with $6$ columns and $5$ rows, in which the numbers $1,2, \ldots , 5$ appear in the first column, $6 , 7, \ldots , 10$ appear in the second column, and so on.  We let $D$ be the corresponding divisor on the generic chain of loops.  Our goal is to show that the full set of functions $\psi_{ijk}$ is tropically independent.  To that end, suppose that the minimum $\min \{ \psi_{ijk} \}$ occurs everywhere at least twice, and let $\theta$ denote this minimum.

As in Lemma \ref{Lemma:FirstColumn}, we see that $\sigma_4 \leq 15$.  If $\sigma_5 > 15$, then the only $\sigma$-permissible functions on $\gamma_5^{\circ}$ are $\psi_{013}$ and $\psi_{022}$.  Since no two functions have identical restrictions to the loop $\gamma_5^{\circ}$, this is impossible, and thus $\sigma_5 \leq 15$.  Now, the same argument as in Lemma \ref{Lemma:SecondColumn} shows that $\sigma_9 \leq 14$.  If $\sigma_{10} > 14$, then the only $\sigma$-permissible functions on $\gamma_{10}^{\circ}$ are $\psi_{023}$ and $\psi_{122}$, which is again impossible.  It follows that $\sigma_{10} \leq 14$.

If we can show that $\sigma_{15} \leq 14$, then by a symmetric argument, we will be done.  To see this, we follow the argument in part (1) of Theorem~\ref{thm:mainthm}.  If $\sigma_{15} > 14$, then the $\sigma$-permissible functions on the block $\Gamma_{[11,15]}$ are those of the form $\psi_{ij2}$, where $i \leq 2$ and $j \geq 2$.  If $\sigma_{13} > 15$, then the $\sigma$-permissible functions on $\Gamma_{[11,13]}$ are
\[
\psi_{024}, \psi_{123}, \psi_{025}, \psi_{124}, \psi_{222} .
\]
But, by the argument in Theorem~\ref{thm:mainthm}, these 5 functions have independent restrictions to $\Gamma_{[11,13]}$.  It follows that $\sigma_{13} \leq 15$.  But there do not exist two functions of the form $\psi_{ij2}$ with $i \leq 2 \leq j$ and $\sigma_{13} \psi_{ij2} \leq 15$ that have the same slope on $\beta_{13}$.  It follows that $\sigma_{15} \leq 14$, and therefore the tropical maximal rank conjecture holds in this case.

\bibliography{math}

\end{document}